\documentclass[10pt]{amsart}

\usepackage{amscd}
\usepackage{amsfonts}
\usepackage{amsmath, amsthm, amssymb, enumerate}
\usepackage{latexsym}

\theoremstyle{plain} 
\newtheorem{thm}{Theorem}
\newtheorem{cor}[thm]{Corollary}
\newtheorem{lem}[thm]{Lemma}
\newtheorem{prop}[thm]{Proposition}

\theoremstyle{definition}
\newtheorem{defn}[thm]{Definition}
\newtheorem{ex}[thm]{Example}
\newtheorem{conj}[thm]{Conjecture}

\theoremstyle{remark}

\newtheorem*{pff}{Proof of Theorem \ref{free}}
\newtheorem*{pfa}{Proof of Theorem \ref{ARC}}

\def\m{\mathfrak m}
\def\p{\mathfrak p}
\def\q{\mathfrak q}
\def\P{\mathfrak P}

\def\Hom{\operatorname{Hom}}
\def\lHom{\operatorname{\underline{Hom}}}
\def\Ext{\operatorname{Ext}}
\def\lExt{\mathrm{\widehat{Ext}}\mathrm{}}

\def\arc{{\sf (ARC)}}

\begin{document}

\title{The Auslander-Reiten conjecture for Gorenstein rings}
\author{Tokuji Araya}
\address{Nara University of Education, Takabatake-cho, Nara-city 630-8528, Japan}
\email{araya@math.okayama-u.ac.jp}

\begin{abstract}
The Nakayama conjecture is one of the most important conjectures in ring theory.
The Auslander-Reiten conjecture is closely related to it.
The purpose of this note is to show that if the Auslander-Reiten conjecture holds in codimension one for a commutative Gorenstein ring $R$, then it holds for $R$.
\end{abstract}
\maketitle
\section{Introduction}

Over thirty years ago, Auslander and Reiten \cite{AR} conjectured the following, which is called {\em the generalized Nakayama conjecture}.

\begin{conj}
Let $\Lambda$ be an artin algebra.
Then any indecomposable injective $\Lambda$-module appears as a direct summand in the minimal injective resolution of $\Lambda$.
\end{conj}

They showed that the above conjecture holds for all artin algebras if and only if the following conjecture holds for all artin algebras.

\begin{conj}\label{ausrei}
Let $\Lambda$ be an artin algebra and $M$ a finitely generated $\Lambda $-module.
If $\Ext ^{i}_{\Lambda }(M,M\oplus \Lambda ) = 0$ for all $i>0$, then $M$ is projective.
\end{conj}

This long-standing conjecture is called {\em the Auslander-Reiten conjecture}.
Auslander and Reiten \cite{AR} proved it for several classes of rings including rings of finite representation type and rings of radical square zero.
About two decades later, Auslander, Ding and Solberg \cite{ADS} investigated it for commutative rings.
They studied the following condition on a ring, not necessarily an artinian ring.

\medskip

\noindent
\arc\quad
Let $R$ be a commutative noetherian ring and $M$ a finitely generated $R$-module.
If $\Ext ^{i}_{R}(M,M\oplus R) = 0$ for all $i>0$, then $M$ is projective.

\medskip

\noindent
They proved that any complete intersection local ring $R$ satisfies this condition.
As a refinement of this result, Araya and Yoshino \cite{AY} showed that every local ring $R$ satisfies the condition {\arc} for $R$-modules $M$ of finite complete intersection dimension.

Conjecture \ref{ausrei} is also called {\em Tachikawa's conjecture} in the case where the ring $\Lambda$ is selfinjective, and it is itself an important problem.
The purpose of this note is to explore Conjecture \ref{ausrei} for Gorenstein rings, which are a common generalization of a complete intersection ring and a commutative selfinjective ring.
To be precise, a commutative noetherian ring $R$ is called {\em Gorenstein} if $R_\p$ has finite injective dimension as an $R_\p$-module for all prime ideals $\p$ of $R$.
The following theorem is the main result of this note.

\begin{thm}\label{ARC}
Let $R$ be a Gorenstein ring.
If $R_\p$ satisfies {\arc} for all prime ideals $\p$ of height at most one, then $R_\p$ satisfies {\arc} for all prime ideals $\p$.
\end{thm}

Gorenstein normal rings, which are complete intersections in
codimension one, are classical and important rings in commutative
algebra and algebraic geometry.  One can find the details in \cite{Hoc} for
example.
Since any normal ring satisfies Serre's condition $(R_1)$ and any regular ring satisfies the condition \arc, we get the following result as a corollary.

\begin{cor}\label{normal}
Every Gorenstein normal ring satisfies the condition {\arc}.
\end{cor}

C. Huneke and G. J. Leuschke proved the following theorem \cite[Theorem 0.1]{HL}.

\begin{thm}[Huneke-Leuschke]\label{HL}
Let $R$ be a Cohen-Macaulay ring which is a quotient of a locally excellent ring $S$ of dimension $d$ by a locally regular sequence. Assume that $S$ is locally a complete intersection ring in codimension one, and further assume either that $S$ is Gorenstein, or that $S$ contains the field of rational numbers. Then $R$ satisfies {\arc}.
\end{thm}

Since complete intersections satisfy {\arc}, Theorem \ref{ARC} actually shows that every Gorenstein ring which is a complete intersection in codimension one satisfies {\arc}, and this does recover the special case of $R = S$ in Theorem \ref{HL}, and without having to assume that $S$ is locally excellent.

In the next section we shall prove Theorem \ref{ARC}.
Our main instruments in the proof are stable cohomology modules.

\section{Proof of Theorem \ref{ARC}}

To prove Theorem \ref{ARC}, notice that it is enough to show the local case.
Throughout the rest of this note, let $R$ be a commutative Gorenstein local ring of Krull dimension $d$ with maximal ideal $\m$.
We begin with recalling the definition of a stable cohomology module.

\begin{defn}
Let $M$ and $N$ be maximal Cohen-Macaulay $R$-modules.
We set $\lHom_R(M,N)= \Hom _R(M,N)/\P_R(M,N)$, where $\P_R(M,N)$ is the $R$-submodule of $\Hom_R(M,N)$ consisting of all homomorphisms from $M$ to $N$ factoring through some finitely generated free $R$-module.
For each integer $i$, we define {\em the $i$-th stable cohomology module} by $\lExt^i_R(M,N) =\lHom_R(\Omega ^i M,N)$.
\end{defn}

We remark that, since $R$ is Gorenstein, every maximal Cohen-Macaulay
$R$-module $M$ admits a {\em complete resolution}, namely, there
exists an exact sequence $(F_\bullet,d_\bullet)$ of finitely generated
free $R$-modules with $M=\operatorname{Im}d_0$:

$$
\begin{picture}(275,50)
\put (0,31){$\cdots$}
\put (16,34){\vector(1,0){36}}
\put (67,34){\vector(1,0){36}}
\put (118,34){\vector(1,0){36}}
\put (175,34){\vector(1,0){36}}
\put (233,34){\vector(1,0){36}}
\put (271,31){$\cdots$}
\put (53,30){$F_1$}
\put (104,30){$F_0$}
\put (155,30){$F_{-1}$}
\put (212,30){$F_{-2}$}

\put (119,29){\vector(1,-1){12}}
\put (142,17){\vector(1,1){12}}
\put (129,8){$M$}

\put (29,40){$d_{2}$}
\put (80,40){$d_{1}$}
\put (131,40){$d_{0}$}
\put (187,40){$d_{-1}$}
\put (245,40){$d_{-2}$}
\end{picture}
$$

\noindent
Then $\Omega^iM$ is defined as the image of the map $d_i$ for every integer $i$.
For the details of stable cohomology modules, see [5] for example.

We make a list of several basic properties of $\lHom$ and $\lExt$.
The details can be found, for instance, in \cite[\S 7]{T}.

\begin{prop}
Let $M,N$ be maximal Cohen-Macaulay $R$-modules.
\begin{enumerate}[\rm (1)]
\item
One has $\lExt ^0_R(M,N) = \lHom _R(M,N)$.
\item
For integers $i$, $m$ and $n$, one has $\lExt ^i_R(M,N) \cong \lExt ^{i-m+n}_R(\Omega^{m}M,\Omega^{n}N)$.
In particular, $\lExt ^i_R(M,N) \cong \Ext ^i_R(M,N)$ for all positive integers $i$.
\item
One has $\lHom_R(M,M)=0$ if and only if $M$ is free.
\end{enumerate}
\end{prop}

We say that a finitely generated $R$-module $M$ is a {\em vector bundle} if $M_\p$ is a free $R_\p$-module for all nonmaximal prime ideals $\p$.

\begin{thm}\label{free}
Let $M$ be a maximal Cohen-Macaulay $R$-module which is a vector bundle.
If $\lExt ^{d-1}_R (M,M) =0$, then $M$ is free. 
\end{thm}

The lemma below, which is a kind of the Auslander-Reiten-Serre duality theorem, is necessary to prove this theorem.
The Auslander-Reiten-Serre duality was essentially given by Auslander \cite[I.8.8,III.1.8]{Aus}.

\begin{lem}\label{ARSdual}
Let $M,N$ be maximal Cohen-Macaulay $R$-modules which are vector bundles.
Then we have an isomorphism
$$
\Ext ^{d}_R(\lHom_R(N,M), R) \cong \lExt ^{d-1}_R(M, N).
$$
\end{lem}

\begin{proof}
It follows from \cite[Lemma (3.10)]{Y}.
\end{proof}

\begin{pff}
Since $M$ is a vector bundle, we have $\lHom_R(M,M)_\p \cong \lHom_{R_\p}(M_\p,M_\p)=0$ for every nonmaximal prime ideal $\p$.
Hence the $R$-module $\lHom _R(M,M)$ has finite length.
The local duality and Lemma \ref{ARSdual} yield isomorphisms $\lHom _R(M,M)\cong \Ext ^{d}_R(\Ext ^{d}_R(\lHom _R(M,M),R),R)\cong \Ext ^{d}_R(\lExt ^{d-1}_R(M,M),R)$.
Therefore if $\lExt ^{d-1}_R(M,M)=0$, then $M$ is free.
\qed
\end{pff}

As an immediate consequence of Theorem \ref{free}, we obtain the following result.

\begin{cor}\label{cor}
Assume $d\ge 2$.
Let $M$ be a maximal Cohen-Macaulay $R$-module which is a vector bundle.
If $\Ext ^{d-1}_R (M,M) =0$, then $M$ is free.
\end{cor}


The following example says that in Theorem \ref{free} and Corollary \ref{cor} the assumption that $M$ is a vector bundle cannot be removed.

\begin{ex}
Let $k$ be a field.
Set $R=k[[x,y,z]]/(xy)$ and $M=R/(x)$.
Then $R$ is a Gorenstein local ring of Krull dimension $2$ and $\Ext ^{2-1}_R(M,M)=0$, but $M$ is not free.
\end{ex}

Now let us prove our main theorem.

\begin{pfa}
Let $\Phi$ be the set of prime ideals $\p$ of $R$ such that $R_\p$ does not satisfy the condition {\arc}.
Assume that $\Phi$ is nonempty, and take a minimal element $\q$ of $\Phi$.
Replacing $R$ with $R_\q$, we may assume that $R$ is a Gorenstein local ring of dimension $d\ge 2$ which does not satisfy {\arc} such that $R_\p$ satisfies {\arc} for all nonmaximal prime ideal $\p$.
Then there exists a nonfree finitely generated $R$-module $M$ such that $\Ext ^{i}_R(M, M\oplus R)=0$ for all $i>0$.
As $\Ext ^{i}_R(M, R)=0$ for $i>0$, $M$ is maximal Cohen-Macaulay.
For any nonmaximal prime ideal $\p$, we have $\Ext ^{i}_{R_\p}(M_\p, M_\p\oplus R_\p)=0$ for every $i>0$ and $R_\p$ satisfies {\arc}, so the $R_\p$-module $M_\p$ is free.
Hence $M$ is a vector bundle.
Since $d-1>0$, we have $\Ext ^{d-1}_R(M,M)=0$.
Corollary \ref{cor} implies that $M$ is a free $R$-module.
Thus we get a contradiction, and the set $\Phi$ must be empty.
\qed
\end{pfa}

\section*{Acknowledgments}
The author would like to express his deep gratitude to Ryo Takahashi, Lars Winther Christensen and Yuji Yoshino, who gave him a lot of valuable comments and suggetions.


 \bibliographystyle{amsplain} 
 
\ifx\undefined\bysame 
\newcommand{\bysame}{\leavevmode\hbox to3em{\hrulefill}\,} 
\fi


\end{document}